\documentclass[12pt]{article}
\usepackage{amssymb,amsmath,amsthm, amsfonts}
\usepackage{graphicx}
\usepackage{epsfig}
\usepackage{tikz}

\def\disp{\displaystyle}

\def\dref#1{(\ref{#1})}

\theoremstyle{plain}
\newtheorem{theorem}{Theorem}[section]
\newtheorem{lemma}{Lemma}[section]

\newtheorem{corollary}{Corollary}[section]

\theoremstyle{definition}

\newtheorem{remark}{Remark}[section]

\numberwithin{equation}{section}

\begin{document}

\title{\bf A new result for
 2D boundedness of solutions to a chemotaxis--haptotaxis model  with/without sub-logistic source}
\author{
Tian Xiang$^{a}$
and Jiashan Zheng$^{b}$\\
$^{a}$   Institute for  Mathematical Sciences,\\
 Renmin University of China, Beijing, 100872, P. R. China \\
$^{b}$  School of Mathematics and Statistics Science,\\
Ludong University, Yantai 264025,  P. R. China
}
\date{}

\maketitle \vspace{0.3cm}
\noindent
\begin{abstract}
We consider  the  Neumann problem for a coupled chemotaxis-haptotaxis model of cancer invasion   with/without kinetic source in a 2D bounded and smooth  domain.  For a large class of cell
kinetic sources including zero source and sub-logistic sources, we detect an explicit condition involving the chemotactic strength, the asymptotic "damping" rate, and the initial mass of cells to ensure uniform-in-time boundedness for the corresponding  Neumann  problem. Our finding
 significantly improves existing 2D global existence and boundedness   in related chemotaxis-/haptotaxis systems.
\end{abstract}

\vspace{0.3cm}


\newpage
\section{Introduction}
Chemotaxis is the motion of cells moving towards the higher concentration of a chemical signal.
A  celebrated minimal mathematical system modelling chemotaxis was initially proposed by Keller and Segel in 1970 (\cite{Keller79}), which, of minimal form, reads as
 \begin{equation*}
 \left\{\begin{array}{ll}
  u_t=\Delta u-\chi \nabla\cdot(u\nabla v), &
x\in \Omega, t>0,\\
  \tau v_t=\Delta v +u- v, &
x\in \Omega, t>0,
 \end{array}\right.\label{1.ssdessdffferrttrrfff1}
\end{equation*}
where $ \chi>0, \tau\geq0$, $ u$ is  the cell density,  $v$ is the chemical concentration, and $\Omega\subset \mathbb{R}^n(n\geq1)$ is a bounded smooth domain.
Since then,  numerous variants of the Keller-Segel system were proposed and have been extensively studied, we refer to the review papers \cite{Bellomo1216,Hillen79,Horstmann2710}
for detailed descriptions of those models and their developments.  The striking future of the KS type chemotaxis model  is the possibility of singularity formation of solutions, which strongly depends on the underlying space dimension and the total mass of cells (\cite{Horstmannff79,Horstmann791,Winkler792, Winkler793}).

To investigate the birth-death effect  of  population,
considerable effort has been devoted to the following    Keller-Segel minimal-chemotaxis-logistic model and its various variants:
\begin{equation}
 \left\{\begin{array}{ll}
  u_t=\Delta u-\chi\nabla\cdot(u\nabla v)+ \kappa u -\mu u^2, &x\in \Omega, t>0,\\
 \tau v_t=\Delta v +u- v, &x\in \Omega, t>0,\\
 \frac{\partial u}{\partial \nu}=\frac{\partial v}{\partial \nu}=0, &x\in \partial\Omega, t>0,\\
 u(x,0)=u_0(x),v(x,0)=v_0(x), &x\in \Omega,
 \end{array}\right.\label{1ssddd.1}
\end{equation}
where $\kappa\in\mathbb{R},  \mu>0$ and  $\frac{\partial}{\partial\nu}$ means the outward  normal derivative on $\partial\Omega$. The  presence of  logistic  source has been shown to have an effect of   blow-up prevention. Indeed, for  $n=1,2$,  even arbitrarily  small $\mu>0$ is enough to prevent blow-up by ensuring all solutions to  \eqref{1ssddd.1} are global-in-time and uniformly bounded for all reasonably initial data  \cite{HP04, JX18-CRAS, Osakix391, Xiang15-JDE}. This is even  true for a two-dimensional chemotaxis system with singular sensitivity \cite{FW14, ZZ17-ZAMP}.  A very recent subtle study  from \cite{Xiang18-JMP} shows   that logistic damping is not the weakest damping to
 ensure boundedness for \eqref{1ssddd.1} in $2$-D. More precisely, with the logistic source $\kappa u -\mu u^2$ in \eqref{1ssddd.1} replaced by a locally bounded kinetic term $f(u)$  satisfying
$f(0) \geq 0$ as well as
\begin{equation}
(\chi-\hat{\mu}_1)^+\hat{M}_1<\frac{1}{2C_{GN}^4},
\label{1.3xcxddfddddffssss29}
\end{equation}
where $\hat{\mu}_1\in [0,\infty]$ and $C_{GN}$ is the Gagliardo-Nirenberg constant and
\begin{equation}
\hat{\mu}_1:=\liminf_{s\rightarrow+\infty}\left\{-f(s)\frac{ \ln s}{s^2}\right\}, \quad \quad
\hat{M}_1=\|u_0\|_{L^1(\Omega)}+|\Omega|\inf_{\eta>0}\frac{\sup\{f(s)+\eta s:s>0\}}{s},
\label{1.3xcxddfffssss29}
\end{equation}
then \dref{1ssddd.1}  admits a unique global-in-time and uniformly-in-time bounded classical solution. Evidently, besides the standard logistic source, $f$ covers sub-logistic sources like:
\begin{equation}\label{sub-log}
f_1(u)=au-\frac{bu^2}{\ln^\gamma (u+1)}\ \  \text{ and } \ \   f_2(u)=au-\frac{bu^2}{\ln(\ln(u+e))}
\end{equation}
for some $a\in\mathbb{R}, b>0,\gamma\in(0,1)$. This provides a further understanding about the chemotactic aggregation induced by $-\chi\nabla\cdot(u\nabla v)$ in \eqref{1ssddd.1} in 2D setting.

In the cases $n\geq3$,  the competition between chemotactic aggregation  and logistic  damping becomes increasingly complicated; for \eqref{1ssddd.1} with $\tau=0$, the effect of logistic  damping  is stronger than that of chemotactic aggregation  when $\mu\geq \frac{(n-2)}{n}\chi$ \cite{Huaffrk,Kangaffrk,Tello710, Xiangssffrk}; for \eqref{1ssddd.1} with $\tau>0$,  the situation that logistic  damping wins over chemotactic aggregation has been studied qualitatively and quantitatively in a series of works under certain largeness on the ratio $\frac{\mu}{\chi}$ \cite{ML16, Winkler37103, Xiang18-JMAA, Xiang18-SIAP,Zhengssdddssddddkkllssssssssdefr23}. For more properties of related chemotaxis models with more complex mechanisms, we refer to \cite{Bellomo1216, Wang79, Win18, Zheng0,Zhengsssddswwerrrseedssddxxss} and the references therein.

For chemotaxis-only systems, our starting motivation  here is to understand further how weak a degradation of cell is needed to suppress  the minimal chemotactic aggregation as appeared in \eqref{1ssddd.1} so that no blow-up can occur in 2D setting. Mathematically, can those sub-logistic restrictions  \eqref{1.3xcxddfddddffssss29} and \eqref{1.3xcxddfffssss29}  or concretely, sub-logistic sources specified in \eqref{sub-log} somehow be further weakened while maintaining 2D global boundedness?

To inspire our second and also primary  motivation, we observe that one important extension of the minimal KS chemotaxis model  to a more complex cell
migration mechanism (known as haptotaxis mechanism) has been introduced by Chaplain and Lolas \cite{Chaplain1,Chaplain7} (see also  Winkler et al.  \cite{Bellomo1216,Tao72}) to describe processes of cancer invasion into surrounding healthy tissue.  In  this context, $u$ represents the density of cancer cell, $v$ denotes the concentration of matrix degrading enzyme (MDE), and $w$ stands for the density of extracellular matrix (ECM). Then $(u,v,w)$  verifies the following no-flux  boundary and initial value problem for the minimal chemotaxis-haptotaxis model:
\begin{equation}
 \left\{\begin{array}{ll}
  u_t=\Delta u-\chi\nabla\cdot(u\nabla v)-\xi\nabla\cdot
  (u\nabla w)+ f(u,w), &
x\in \Omega, t>0,\\
 \tau v_t=\Delta v +u- v, &
x\in \Omega, t>0,\\
 w_t=- vw, &
x\in \Omega, t>0,\\
 \frac{\partial u}{\partial \nu}-\chi u\frac{\partial v}{\partial \nu}-\xi w\frac{\partial w}{\partial \nu}=\frac{\partial v}{\partial \nu}=0, &
x\in \partial\Omega, t>0,\\
 u(x,0)=u_0(x),\tau v(x,0)=\tau v_0(x),w(x,0)=w_0(x), &
x\in \Omega,
 \end{array}\right.\label{1.1}
\end{equation}
 where  $\chi$ and  $\xi > 0$ measure the chemotactic
and haptotactic sensitivities respectively, and, $f(u,w)$  characterizes the proliferation and death
of cancer cells including competition for space with the ECM. As for the initial data $(u_0, \tau v_0,w_0)$, for convenience, we assume throughout this paper,  for some $\vartheta\in(0,1)$ and $A\geq0$, that
\begin{equation}\label{x1.731426677gg}
\left\{
\begin{array}{ll}
 u_0\in C(\bar{\Omega})~~\mbox{with}~~u_0\geq0~~\mbox{in}~~\Omega~~\mbox{and}~~u_0\not\equiv0,\\
 \tau v_0\in W^{1,\infty}(\Omega)~~\mbox{with}~~\tau v_0\geq0~~\mbox{in}~~\Omega,\\
 w_0\in C^{2+\vartheta}(\bar{\Omega})~~\mbox{with}\ \ w_0\geq0, \ \ |\nabla w_0|^2\leq A w_0 \ \ \mbox{in}~~\Omega~~\mbox{and}~~\frac{\partial w_0}{\partial\nu}=0~~\mbox{on}~~\partial\Omega. \\
\end{array}
\right.
\end{equation}
For the  commonly chosen logistic-type competition source $f$:
\begin{equation}\label{f-log}
f(u,w)=\mu u (1-u-w), \quad \mu>0,
\end{equation}
the global solvability, boundedness and asymptotic behavior for  models of type  \eqref{1.1} has been widely explored.  For haptotaxis-only models, i.e., $\chi=0$, the global existence and boundedness are investigated  in \cite{MCP10, MR08, TZ07,  WW07}   and
asymptotic behavior of solution is studied in \cite{LM10} with/without logistic source.  In the parabolic-elliptic case, i.e.,  $\tau = 0$,    Tao and Wang \cite{Tao2} proved  the global existence and  boundedness of  classical solutions to \eqref{1.1}  for any $\mu>0$ in 2D, and for large $\mu > 0$ in 3D; later on, Tao and Winkler subsequently studied global boundedness for
model \dref{1.1}  under the condition $\mu>\chi$ \cite{Taox26} and $\mu> \frac{(n-2)^+}{n}\chi$ \cite{Taox26216, TW15-SIAM}, and  also gave the exponential decay of $w$  under  additional   smallness on $w_0$; in the parabolic-parabolic case, i.e., $\tau=1$, Tao and Wang \cite{Tao3}  proved that the model \dref{1.1}  possesses a unique global-in-time classical solution for any $\chi>0$ in 1D, and for large $\frac{\mu}{\chi}$ in 2D; the latter was  improved to any $\mu>0$ by Tao \cite{Taoss}. In 3D, the global boundedness was obtained by Cao \cite{Caoss} for large $\frac{\mu}{\chi}$.  These are the main progressive developments on the minimal chemotaxis-haptotaxis model \dref{1.1}. While, we would like to mention  there appears a rapidly growing literature on a general framework of \dref{1.1} with more complex mechanisms like nonlinear diffusion,  remodeling effects  and generalized logistic sources etc;, it reads essentially as
\begin{equation}
 \left\{\begin{array}{ll}
  u_t=\nabla\cdot((u+\epsilon)^m\nabla u)-\chi\nabla\cdot(u^\alpha\nabla v)-\xi\nabla\cdot
  (u^\beta\nabla w)+ \mu u^\gamma(1-u-w),\\
 \tau v_t=\Delta v +u^\delta- v,\\
 w_t=- vw+\eta w(1-w-u)
 \end{array}\right.\label{CH-gen}
\end{equation}
for some given nonnegative parameters $\epsilon, \eta, \tau,  \alpha, \beta, \gamma,\delta, \chi, \xi, \mu$ and $m\in\mathbb{R}$. For effects of  possible interactions between various mechanisms on  dynamical properties of the corresponding IBVP for model \eqref{CH-gen} and its variants, we refer the interested reader  to
\cite{Jindd3344455678887,Zheng445ss55,  Li445666, liujinijjkk1, PangPang1, PangPang2,PW19, Tao1, Tao11, Tao72,Tao79477,TW19, Wanghhss,
Zhddengssdeeezseeddd0,Zheng1111445ss55, Zheng44555}.

By comparison, we find  that results on chemotaxis-/haptotaxis systems (at least, for the minimal case like \eqref{1ssddd.1} and \eqref{1.1}) are similar to that of  chemotaxis-only systems obtained upon setting $w\equiv0$; roughly, in $2$-D, any presence of logistic source will be sufficient to rule out blow-up and strong logistic damping can prevent blow-up in $\geq 3$-D. In this work,  based on \cite{Xiang18-JMP}, we are wondering mainly from mathematical point of view, in the already complex minimal chemotaxis-haptotaxis  model \eqref{1.1}, whether or not a weaker damping source  than  the standard  logistic source \eqref{f-log} is still sufficient to ensure 2D global  boundedness so as to provide further understanding  about the minimal chemotactic aggregation   in \eqref{1.1} vs the kinetic source $f$. In the future, we shall consider providing further understanding about the chemotactic aggregation   in \eqref{CH-gen} vs generalized logistic sources and other ingredients. Our finding on the chemotactic aggregation in \eqref{1.1} vs the kinetic  source $f$ including zero source, logistic source like \eqref{f-log} and sub-logistic source is captured in our following main theorem:
\begin{theorem}\label{theorem3}Let $\chi,\xi>0$, $\tau\geq0$, $\Omega\subset\mathbb{R}^2$ be a bounded and smooth   domain,  the initial data $(u_0, \tau v_0,w_0)$ satisfy \eqref{x1.731426677gg} and,  finally,  let the locally bounded source $f$  satisfy $f (0,w)\geq0$ as well as
 \begin{equation}\label{f-source}
 \text{either } f\equiv 0 \text{ or } \left\{\exists a\in \mathbb{R}, b>0 \ s.t.\   f(s,w)\leq a-bs \text{ on } (0,+\infty)\times (0,\max_{x\in\bar{\Omega}}w_0(x))\right\}.
 \end{equation}
 For positive integer $r\geq 1$, we define the extended  asymptotic "damping" rate $\mu_r$ as
\begin{equation}
\mu_r:=\liminf_{s\rightarrow+\infty}\left\{\inf_{0\leq w\leq \max_{x\in\bar{\Omega}}w_0(x)}\left\{-f(s,w)\frac{\prod_{i=1}^r \ln^{[i]}s}{s^2}\right\}\right\},
\label{1.3xcxddfff29}
\end{equation}
where and below, as usual,  $\ln^{[i]}(s)$ denotes the $i$-th iterate of $\ln(s)$. Assume that
  \begin{equation}\label{thm-con}
\text{either } \Bigr\{\tau=0, \ \exists r\geq 1, \  s.t. \  \mu_r\in(0, +\infty] \Bigr\} \text{ or }    (\chi-\mu_1)^+M_1<\frac{1}{2C_{GN}^4},
  \end{equation}
  where  $C_{GN}$ is the Gagliardo-Nirenberg  constant (cf. \eqref{cz2.51ssdsdddddddfjfiiid14114})  and  $M_1$ is finite and is  given by
  \begin{equation}\label{M-def}
  M_1=\|u_0\|_{L^1(\Omega)}+|\Omega|\inf_{\eta\in(0, b]}\left\{\begin{array}{ll}
  0, &\text{if } f\equiv0,\\
 \frac{\sup\left\{ f(s,w)+\eta s:\  (s,w)\in(0,\infty)\times(0,\max_{x\in\bar{\Omega}}w_0(x)) \right\}}{\eta},&\text{if } f\not\equiv0.
 \end{array}\right.
\end{equation}
Then there exists a unique nonnegative solution triple $(u,v,w)\in (C^{2,1}
(\bar{\Omega}\times(0,\infty)))^3$ which solves \dref{1.1} classically. Moreover, there exists $C=C(u_0, \tau v_0, w_0, |\Omega|, \chi,\xi, f)>0$ such that
\begin{equation}
\| u(\cdot, t)\|_{L^\infty(\Omega)}+\| v(\cdot, t)\|_{W^{1,\infty}(\Omega)}+\| w(\cdot, t)\|_{W^{1,\infty}(\Omega)}\leq C, \quad \forall t>0.
\label{uvw-bdd}
\end{equation}
\end{theorem}
Before proceeding, there are a few remarks in order.
\begin{remark}
(i) For  logistic  or  super-logistic sources like $f(u,w)=u(a-bu^\theta-w)$ with $a\in\mathbb{R}, b>0$ and $\theta\geq1$ or   sub-logistic sources like
$f(u,w)=u\left(a-w- \frac{b u}{\ln^\gamma  (u+1)}\right)$ with $a\in \mathbb{R}, b>0, \gamma\in(0,1)$ or $ f(u,w)=u\left(a-w- \frac{bu}{\ln(\ln (u+e))}\right)$ with $ a\in \mathbb{R}, b>0$,  one can easily compute  from \eqref{1.3xcxddfff29} that $\mu_1=+\infty$ and so \eqref{thm-con} holds trivially. Therefore, no matter $\tau=0$ or $\tau>0$, the global boundedness for  \eqref{1.1} in 2D is ensured  for all reasonable initial data.

 For $k\geq 1$, we write $e(x)=e^x, e^{[k]}=e^{[k]}(1)$, then, for  a family of sub-logistic sources like
\begin{equation}\label{f-r-def}
f_k(u,w)=u\left(1- w-\frac{\mu u}{\prod_{i=1}^k\ln^{[i]}(u+e^{[i-1]})}\right), \ \ \mu>0,
\end{equation}
we   compute from \eqref{1.3xcxddfff29} that  $\mu_1=\mu_2=\cdots=\mu_{k-1}=0$, $\mu_k=\mu$  and $\mu_r=+\infty$  for any $r>k$. This shows, for $\tau=0$, that the first case of   \eqref{thm-con} holds, and so, we get  the global boundedness for  \eqref{1.1} with $\tau=0$  in 2D   for all reasonable initial data. From these observations, we see that Theorem \ref{theorem3}  improves known 2D global existence and  boundedness with logistic sources  (cf. \cite{Taoss,  Tao3, Tao2, Taox26216, Taox26,TW15-SIAM}) to sub-logistic sources.

While, we have to point out, for $f$ as given in \eqref{f-r-def} and $\tau>0$, we do not know whether or not the global boundednesss for  \eqref{1.1}   in 2D  holds for large initial data. More worse, for  simple  sub-logistic sources like
$$
f(u,w)=u\left(1- w-\frac{\mu u}{ \ln^2(u+1)}\right), \ \ \mu>0,
$$
we are unable to conclude whether or not  the  global boundednesss for  \eqref{1.1} in 2D holds for large initial data, even when $\tau=0$ and $w\equiv0$.

(ii) In the chemotaxis-only case,  i.e.,   $w\equiv0$,  which is allowed by the assumption of $w_0$  in \eqref{x1.731426677gg}, we see that $\hat{\mu}_1$ defined in  \eqref{1.3xcxddfddddffssss29} is simply $\mu_1$ by setting $r=1$ in \eqref{1.3xcxddfff29}, and, when $\tau=0$,    we see that \eqref{thm-con}   greatly  relaxes \eqref{1.3xcxddfddddffssss29} by allowing more weaker damping sources like $f_k(u,0)$ with $f_k(u,w)$ given by \eqref{f-r-def}, which are much weaker than \eqref{sub-log}. Consequently,   Theorem \ref{theorem3} also  improves     \cite[Theorem 1.1]{Xiang18-JMP}.

(iii) In the haptotaxis-only case, i.e., $\chi=0$, we see that \eqref{thm-con} holds automatically, and thus we obtain global boundedness of classical solutions to \eqref{1.1} with/without growth source for large initial data. This goes beyond  global existence in \cite{MR08, TZ07,  WW07}.
\end{remark}

 We note that, under the basic condition \eqref{f-source} which  entails the uniform $L^1$-boundedness of $u$,  the extended damping rate $\mu_r$ defined in \eqref{1.3xcxddfff29} or $\hat{\mu}_1$ defined in  \eqref{1.3xcxddfddddffssss29} is nonnegative. In 2D chemotaxis-related systems as we have here, the key is how to raise the easily obtained $L^1$-boundedness  of $u$  to $L^1$-boundedness  of $g(u)$ with $\lim_{s\rightarrow+\infty}\frac{|g(s)|}{s}=+\infty$.  In the case of $\tau>0$ and weak damping, the damping effect of diffusion is stronger than that of  kinetic source. To  make use of the diffusion effect, we could only test the $u$-equation by $\ln u$   to derive the uniform $L^1$-boundedness of $u\ln u$, this is why we have to specify $r=1$  in the second condition of \eqref{thm-con}. While, in the case of $\tau=0$, we can study  the evolution of  a slowly  growing function  $(u+e^{[m]})\ln^{[m]}(u+e^{[m]})$ and use the damping  source to establish the uniform $L^1$-boundedness of $(u+e^{[m]})\ln^{[m]}(u+e^{[m]})$ (cf. Lemmas \ref{key-energy-id}, \ref{lemma45630223116} and \ref{lemma45630223}). Upon such improved regularity for $u$ together with a generalization of the logarithmic version of  Gagliardo-Nirenberg inequality (c.f. Lemma \ref{lemma45630jjjj223}), we derive the  $(L^2, L^4)$-boundedness of $(u, \tau\nabla v)$ (cf. Lemma \ref{lemma4ssdffff5630223}). Finally, we use the widely known smoothing $L^p$-$L^q$-estimates for the Neumann heat semigroup in the case of $\tau>0$ and the well-known $W^{2,p}$-regularity theory in the case of $\tau=0$ to conclude the ($L^\infty, W^{1,\infty},W^{1,\infty})$-boundedness of $(u, v,w)$ (c.f. the proof of  Theorem \ref{theorem3}).

\section{Preliminaries and basic results on \eqref{1.1}}

For convenience, we collect the widely used well-known Gagliardo-Nirenberg  inequality.
\begin{lemma}\label{GN-inter}(cf.  \cite{Ishida, Ladyzenskaja710,  Li445666}) Let  $\Omega\subset\mathbb{R}^n(n\geq 1)$ be a bounded smooth domain  and let  $p\geq 1$ and  $q\in (0,p)$. Then there exists a positive constant  $C_{GN}=C(p,q, n, \Omega)$ such that
 $$
 \|w\|_{L^p(\Omega)} \leq C_{GN}\Bigr(\|\nabla w\|_{L^2(\Omega)}^{\delta}\|w\|_{L^q(\Omega)}^{(1-\delta) }+\|w\|_{L^r(\Omega)}\Bigr), \quad \forall w\in H^1(\Omega)\cap L^q(\Omega),
 $$
where $r>0$ is arbitrary and  $\delta$ is given by
 $$
 \frac{1}{p}=\delta(\frac{1}{2}-\frac{1}{n})+\frac{1-\delta}{q}\Longleftrightarrow \delta=\frac{\frac{n}{q}-\frac{n}{p}}{1-\frac{n}{2}+\frac{n}{q}}\in(0,1).
 $$
\end{lemma}
 The  local  solvability and extendibility   of classical solutions to  the chemotaxis-hapotataxis  system \eqref{1.1} is quite standard; see analogous discussions in  \cite{MR08, Tao72, Tello710, Winkler37103, liujinijjkk1}.
\begin{lemma}\label{lemma70}
Let $\chi,\xi>0$, $\tau\geq0$, $\Omega\subset\mathbb{R}^n$ be a bounded and smooth   domain,  the initial data $(u_0, \tau v_0,w_0)$ satisfy \eqref{x1.731426677gg} and  let the locally bounded source $f$  satisfy $f (0,w)\geq0$. Then there exists a maximal existence time $T_m\in(0,\infty]$ and a unique triple $(u,v,w)$ of functions from $C^0(\bar{\Omega}\times[0,T_m))\cap C^{2,1}(\bar{\Omega}\times(0,T_m))$ solving \dref{1.1}  classically in $\Omega\times(0,T_m)$ and such that
\begin{equation}\label{uvw-nonegative}
0< u, \ \ \ 0< v, \ \ \ 0 \leq w\leq \|w_0\|_{L^\infty(\Omega)}.
\end{equation}
Moreover, we have the following extendibility alternatives:
\begin{equation}
\text{either } T_m=+\infty \text{ or  }\limsup_{t\rightarrow T_m-}\left(\|u(\cdot, t)\|_{L^\infty(\Omega)}+\|v(\cdot,t)\|_{W^{1,\infty}(\Omega)}\right)=+\infty.
\label{1.163072x}
\end{equation}
\end{lemma}
Henceforth, we assume that all the conditions in Lemma \ref{lemma70} and Theorem \ref{theorem3} are satisfied. $C$, $C_i$ (numbering within lemmas or theorems) and $C_\epsilon$ etc will denote some generic constants which may vary line-by-line. Also, the integration variable will be omitted. Now,  we start with the $(L^1, L^2)$-boundedness of $(u,v)$, which is collected in the following lemma:
\begin{lemma}\label{ul1-bdd} Under \eqref{f-source}, the $(L^1, L^2)$-norm of $(u,v)$ is  uniformly bounded according to
\begin{equation}\label{ul1-est-sub}
\| u(\cdot,t)\|_{L^1(\Omega)}\leq M_1, \quad \forall t\in(0,T_m),
\end{equation}
where $M_1$ defined by \eqref{M-def}, and, there exists $C=C(u_0,\tau v_0, w_0, |\Omega|, f)>0$ such that
\begin{equation}\label{vl2-est-sub}
\| v(\cdot,t)\|_{L^2(\Omega)}\leq C, \quad \forall t\in(0,T_m).
\end{equation}
\end{lemma}
\begin{proof}We only show the proof for the case of $f\not\equiv0$. Integrating the $u$-equation in \eqref{1.1} and using the no flux boundary conditions and \eqref{f-source},  we obtain an ordinary  differential inequality (ODI), for any $\eta\in (0,b]$ and for any $t\in(0, T_m)$, that
\begin{equation*}
\frac{d}{dt}\int_\Omega u= \int_\Omega f(u,w)\leq -\eta \int_\Omega u+M_\eta|\Omega|,
\end{equation*}
which trivially yields
$$
\int_\Omega u\leq \int_\Omega u_0+\frac{M_\eta}{\eta}|\Omega|.
$$
Upon taking infimum over $\eta\in (0, b]$ and recalling the definition of $M_1$ in \eqref{M-def},  we infer the $L^1$-bound of $u$ as stated in \eqref{ul1-est-sub}. Here, since $\eta\in(0,b]$ it follows from  \eqref{f-source} and \eqref{uvw-nonegative} that
$$
M_\eta=\sup\Bigr\{ f(s,w)+\eta s:\  (s,w)\in(0,\infty)\times(0,\max_{x\in\bar{\Omega}}w_0(x))\Bigr\}<\infty.
$$
Now, since $\|u\|_{L^1}$ is bounded, the $L^1$-boundedness of $v$ follows from
$$
\tau\frac{d}{dt}\int_\Omega v+\int_\Omega v=\int_\Omega u\leq M_1.
$$
 When $\tau=0$, the $L^1$-boundedness of $u$ and the elliptic estimate applied to the $v$-equation yield easily  the $L^2$-boundedness of  $v$.  When $\tau>0$, we rewrite  the $v$-equation in \dref{1.1} as
  $$
  v(t)=e^{\frac{t}{\tau}(\Delta-1)}v_0+\frac{1}{\tau}\int_0^t e^{\frac{(t-s)}{\tau}(\Delta-1)}u(s)ds,
  $$
  and then use the quite known smoothing $L^p$-$L^q$-estimates for the Neumann heat semigroup  $\{e^{t\Delta}\}_{t\geq0}$ in $\Omega$ (cf.   \cite{Cao15, Horstmann791, Winkler792}) to derive the following reciprocal estimate (cf., \cite{JX16, Li445666, Xiang15-JDE, Xiang18-SIAP})
  $$
  \|v(\cdot,t)\|_{W^{1,1}(\Omega)}\leq C_1\left(1+\sup_{s\in(0,t)}\|u(\cdot,s)\|_{L^1(\Omega)}\right)\leq C_2,
  $$
  which  gives the $L^2$-boundedness of $v$ by the embedding $W^{1,1}(\Omega)\hookrightarrow L^2(\Omega)$  for $\Omega\subset\mathbb{R}^2$.
\end{proof}
It follows from the ODE of $w$ in \eqref{1.1} that $w(x,t)=w_0(x)e^{-\int_0^tv(x,s)ds}$. Then  repeating the argument in \cite[Lemma 2.3]{Taoss} and noting our assumption on $w_0$ in \eqref{x1.731426677gg} that $|\nabla w_0|^2\leq Aw_0$ on $\Omega$,  we obtain a one-sided pointwise estimate for $-\Delta w$ as follows.
\begin{lemma} The local-in-time solution  $(u, v,w)$ of \dref{1.1} fulfills
 \begin{equation}\label{x1.731426677gghh}
 \begin{array}{rl}
-\Delta w(x, t) \leq\tau \|w_0\|_{L^\infty(\Omega)}\cdot v(x,t)+\kappa,  \quad \quad \forall (x,t)\in\Omega\times (0, T_m),
\end{array}
\end{equation}
where
$$
\kappa:=\|\Delta w_0\|_{L^\infty(\Omega)}+4\|\nabla\sqrt{w_0}\|_{L^\infty(\Omega\cap\{w_0>0\})}^2
+\frac{\|w_0\|_{L^\infty(\Omega)}}{e}\leq \|\Delta w_0\|_{L^\infty(\Omega)}+4A
+\frac{\|w_0\|_{L^\infty(\Omega)}}{e}.
$$
\end{lemma}
\section{Bootstrap argument and the proof of Theorem \ref{theorem3}}
\subsection{Bootstrap procedure for improving $L^1$ of $u$ }
In this subsection, we aim to to improve  the starting $L^1$-boundedness of $u$. Our key idea to this end relies on the following  dissipation identity.
\begin{lemma}
\label{key-energy-id}
Let $h:(0,+\infty)\rightarrow \mathbb{R}$ be $C^2$-smooth and let $k\geq 0$. Then  the  unique local-in-time solution of \dref{1.1} satisfies, for $t\in(0, T_m)$,
\begin{equation}\label{key-id}
\begin{split}
&\frac{d}{dt}\int_\Omega (u+k)h(u+k)+\int_\Omega \left( 2h^{'}(u+k)+(u+k)h^{''}(u+k) \right)|\nabla u|^2\\
&=\chi\int_\Omega \nabla\left( u(u+k)h^{'}(u+k)-k[h(u+k)-h(k)] \right) \nabla v\\
&\ \  +\xi\int_\Omega\nabla\left( u(u+k)h^{'}(u+k)-k[h(u+k)-h(k)] \right) \nabla w\\
&\  \  +\int_\Omega \left(h(u+k)+(u+k)h^{'}(u+k)\right)f(u,w)\\
&=-\chi\int_\Omega \left( u(u+k)h^{'}(u+k)-k[h(u+k)-h(k)] \right)\Delta v\\
&\ \ -\xi\int_\Omega \left( u(u+k)h^{'}(u+k)-k[h(u+k)-h(k)] \right) \Delta w\\
&\   \ +\int_\Omega \left(h(u+k)+(u+k)h^{'}(u+k)\right)f(u,w)\\
&=\chi\int_\Omega \left( u(u+k)h^{'}(u+k)-k[h(u+k)-h(k)] \right)\left(u-v-\tau v_t\right)\\
&\ \  -\xi\int_\Omega \left( u(u+k)h^{'}(u+k)-k[h(u+k)-h(k)] \right)\Delta w\\
&\ \ +\int_\Omega \left(h(u+k)+(u+k)h^{'}(u+k)\right)f(u,w).
\end{split}
\end{equation}
In particular, formally setting $h(u)=\ln u$, we have
\begin{equation}\label{key-id-ulnu}
\begin{split}
&\frac{d}{dt}\int_\Omega u\ln u+4\int_\Omega |\nabla u^\frac{1}{2}|^2\\
&=\chi\int_\Omega \nabla u \nabla v+\xi\int_\Omega\nabla u   \nabla w +\int_\Omega \left(\ln u+1\right)f(u,w)\\
&=-\chi\int_\Omega   u \Delta v-\xi\int_\Omega  u  \Delta  w +\int_\Omega \left(\ln u+1\right)f(u,w)\\
&=\chi\int_\Omega u\left(u-v-\tau v_t\right) -\xi\int_\Omega u\Delta w +\int_\Omega \left(\ln u+1\right)f(u,w).
\end{split}
\end{equation}
\end{lemma}
\begin{proof}
Using the no flux boundary conditions and the equations in \dref{1.1}, we calculate that
\begin{align*}
&\frac{d}{dt}\int_\Omega (u+k)h(u+k)=\int_\Omega \left(h(u+k)+(u+k)h^{'}(u+k)\right)u_t\\
&=\int_\Omega\left(h(u+k)+(u+k)h^{'}(u+k)\right)\nabla\cdot\left(\nabla u-\chi u\nabla v-\xi u\nabla w\right)\\
&\ \  +\int_\Omega \left(h(u+k)+(u+k)h^{'}(u+k)\right)f(u,w)\\
&=-\int_\Omega \left( 2h^{'}(u+k)+(u+k)h^{''}(u+k) \right)|\nabla u|^2\\
&\ \ \ +\chi\int_\Omega \left( 2h^{'}(u+k)+(u+k)h^{''}(u+k) \right)u\nabla u\nabla v\\
&\ \ +\xi\int_\Omega \left( 2h^{'}(u+k)+(u+k)h^{''}(u+k) \right)u\nabla u\nabla w\\
&\ \ \ +\int_\Omega \left(h(u+k)+(u+k)h^{'}(u+k)\right)f(u,w)
\end{align*}
and that
\begin{align*}
&\chi\int_\Omega \left( 2h^{'}(u+k)+(u+k)h^{''}(u+k) \right)u\nabla u\nabla v\\
&+\xi\int_\Omega \left( 2h^{'}(u+k)+(u+k)h^{''}(u+k) \right)u\nabla u\nabla w\\
&=\chi\int_\Omega \nabla\left( u(u+k)h^{'}(u+k)-k[h(u+k)-h(k)] \right) \nabla v\\
&\ +\xi\int_\Omega\nabla\left( u(u+k)h^{'}(u+k)-k[h(u+k)-h(k)] \right) \nabla w\\
&=-\chi\int_\Omega \left( u(u+k)h^{'}(u+k)-k[h(u+k)-h(k)] \right)\Delta v\\
&\  -\xi\int_\Omega \left( u(u+k)h^{'}(u+k)-k[h(u+k)-h(k)] \right) \Delta w\\
&=\chi\int_\Omega \left( u(u+k)h^{'}(u+k)-k[h(u+k)-h(k)] \right)\left(u-v-\tau v_t\right)\\
&\ -\xi\int_\Omega \left( u(u+k)h^{'}(u+k)-k[h(u+k)-h(k)] \right) \Delta w.
\end{align*}
Combining these two identities, we arrive at \dref{key-id}.
\end{proof}

 Based on the starting $L^1$-boundedness of $u$, in 2D framework, the next common step is to establish the $L^1$-boundedness of $u\ln u$, a common choice in the literature (\cite{Osakix391, Taoss, Xiang15-JDE, Xiang18-JMP})  for such purpose is based on \dref{key-id-ulnu} via  $h(z)=\ln z$ and $k=0$, which readily entails
$$
\lim_{z\to +\infty} h(z)=+\infty, \quad \quad   2h^{'}(z+k)+zh^{''}(z+k)=\frac{1}{z}
$$
so that the diffusion-induced good terms help one to control  taxis-induced bad terms in \dref{key-id-ulnu}. Here, we shall first  choose a $C^2$-smooth test function $h$  growing slower than  $\ln z$ with the properties that
$$
\lim_{z\to +\infty} h(z)=+\infty, \quad \quad   2h^{'}(z+k)+zh^{''}(z+k)\geq 0, \ \   \forall z>0
$$
so that diffusion is harmless and then we use the damping term $f$ to control the taxis-induced bad terms. The following computation is made out of this purpose.
\begin{lemma}\label{lemma45630223116}
For integer  $m\geq 1$, we have, for any $z>0$, that
 \begin{equation}
\label{hmprime}
\left(\ln^{[m]}(z+e^{[m]})\right)^{'}=\left(\prod_{i=0}^{m-1}\ln^{[i]}(z+e^{[m]})\right)^{-1}>0 \end{equation}
and
 \begin{equation}
\label{hmmprime}
\begin{split}
 &2\left(\ln^{[m]}(z+e^{[m]})\right)^{'}+(z+e_m)\left(\ln^{[m]}(z+e^{[m]})\right)^{''}\\
 &=\left(\prod_{i=0}^{m-1}\ln^{[i]}(z+e^{[m]})\right)^{-1}\left(1-\sum_{k=1}^{m-1}
 \prod_{i=1}^{k}\left(\ln^{[i]}(z+e^{[m]})\right)^{-1}\right)>0.
 \end{split}  \end{equation}
\end{lemma}
Here and below, $e^{[m]}=e^{[m]}(1)$ with $e(s)=e^s$ so that $\ln^{[m]}e^{[m]}=1$ and, the last term on the right-hand side of  \dref{hmmprime} is understood to be void when $m=1$.
\begin{proof} For $m\geq 1$, using product and chain rule, we first compute \dref{hmprime} and
$$
\left(\prod_{i=0}^{m-1}\ln^{[i]}(z+e^{[m]})\right)^{'}
=1+\sum_{k=1}^{m-1}\prod_{i=k}^{m-1}\ln^{[i]}(z+e^{[m]});
$$
then we find
\begin{align*}
&2\left(\ln^{[m]}(z+e^{[m]})\right)^{'}+(z+e^{[m]})\left(\ln^{[m]}(z+e^{[m]})\right)^{''}\\
&=2\left(\prod_{i=0}^{m-1}\ln^{[i]}(z+e^{[m]})\right)^{-1}
 -(z+e^{[m]})\left(\prod_{i=0}^{m-1}\ln^{[i]}(z+e^{[m]})\right)^{-2}\left(1
 +\sum_{k=1}^{m-1}\prod_{i=k}^{m-1}\ln^{[i]}(z+e^{[m]})\right).
 \end{align*}
Since   $\ln^{[i]}(e^{[m]})=e^{[m-i]}>1, i=1,2,\cdots, m-1$, we  further compute that
\begin{align*}
&\left[2\left(\ln^{[m]}(z+e^{[m]})\right)^{'}+(z+e^{[m]})\left(\ln^{[m]}(z+e^{[m]})\right)^{''}\right]\prod_{i=0}^{m-1}\ln^{[i]}(z
+e^{[m]})\\
&=1-\sum_{k=1}^{m-1}
 \prod_{i=1}^{k}\left(\ln^{[i]}(z+e^{[m]})\right)^{-1}\\
&=1-\left(\frac{1}{\ln^{[1]}(z+e^{[m]})}+\frac{1}{\prod_{i=1}^{2}\ln^{[i]}(z+e^{[m]})}+\cdots
+\frac{1}{\prod_{i=1}^{m-1}\ln^{[i]}(z+e^{[m]})}\right)\\
&\geq 1-\left(\frac{1}{\ln^{[1]}(e^{[m]})}+\frac{1}{\prod_{i=1}^{2}\ln^{[i]}(e^{[m]})}+\cdots
+\frac{1}{\prod_{i=1}^{m-1}\ln^{[i]}(e^{[m]})}\right)\\
&\geq 1-\frac{m-1}{e^{[m-1]}}>0,
 \end{align*}
 which shows the desired result \dref{hmmprime}.
\end{proof}

With Lemmas \ref{key-energy-id}  and \ref{lemma45630223116} at hand, we now  improve the $L^1$- regularity of solutions.

\begin{lemma}\label{lemma45630223} Let $r\geq1$ satisfy  \dref{thm-con} and let
\begin{equation}\label{x1.73sdfghh1426677gg}
g(u)=\left\{\begin{array}{ll}
 (u+e^{[r+1]})\ln^{[r+1]}(u+e^{[r+1]}), &\text{ if  }  \tau=0,\\
  u\ln u, &\text{if  }   \tau>0.\\
 \end{array}\right.
 \end{equation}
Then there exists $C>0$ such that  the corresponding solution of \dref{1.1} satisfies
\begin{equation}
\int_{\Omega}\left[|g(u)|+\frac{\tau\chi}{{2}}|\nabla v |^{{2}}\right](\cdot, t)\leq C, \quad \quad \forall t\in(0,T_m).
\label{zjscz2.5297x96302222114}
\end{equation}
\end{lemma}
\begin{proof}
\textbf{Case I:  $\tau= 0$.} In this case,  setting  $m=r+1$ for  consistency with Lemma \ref{lemma45630223116} and taking $h(u) =\ln^{[m]}(u)$ and $k=e^{[m]}$ in Lemma \ref{key-energy-id},   we first see that $g''(u)>0$, and then from  computations \dref{key-id}, \dref{hmprime} and \dref{hmmprime}, we obtain, for $t\in(0, T_m)$, that
\begin{equation}
\begin{split}
&\frac{d}{dt}\int_{\Omega}g(u)+\int_{\Omega}g''(u)|\nabla u|^2
\\
&=\chi\int_\Omega \left(u\left(\prod_{i=1}^{m-1}\ln^{[i]}(u+e^{[m]})\right)^{-1}-e^{[m]}\left(\ln^{[m]}(u+e^{[m]})
-1\right)\right)(u-v)\\
&\ \ -\xi\int_\Omega \left(u\left(\prod_{i=1}^{m-1}\ln^{[i]}(u+e^{[m]})\right)^{-1}-e^{[m]}\left(\ln^{[m]}(u+e^{[m]})
-1\right)\right)\Delta w\\
&\ \ +\int_\Omega \left(\ln^{[m]}(u+e^{[m]})+\left(\prod_{i=1}^{m-1}\ln^{[i]}(u+e^{[m]})\right)^{-1}\right)  f(u,w).
\end{split}
\label{cz2.51ssddd14114}
\end{equation}
We notice from \dref{key-id} in Lemma \ref{key-energy-id} that
$$
u\left(\prod_{i=1}^{m-1}\ln^{[i]}(u+e^{[m]})\right)^{-1}-e^{[m]}\left(\ln^{[m]}(u+e^{[m]})
-1\right)=\int_0^uzg''(z)dz>0,
$$
and then  we employ  the nonnegativity of $u, v$,  $\chi$ and $\xi$ and the one-sided   pointwise estimate of $-\Delta w$ in \dref{x1.731426677gghh}  to infer from \dref{cz2.51ssddd14114}  that
\begin{equation}
\begin{split}
&\frac{d}{dt}\int_{\Omega}g(u)+\int_\Omega g''(u)|\nabla u|^2+\chi e^{[m]}\int_\Omega \left(\ln^{[m]}(u+e^{[m]})
-1\right)u\\
&\ \  \ +\kappa\xi e^{[m]}\int_\Omega \left(\ln^{[m]}(u+e^{[m]})
-1\right)\\
&\leq \chi\int_\Omega u^2\left(\prod_{i=1}^{m-1}\ln^{[i]}(u+e^{[m]})\right)^{-1}+\kappa\xi\int_\Omega u\left(\prod_{i=1}^{m-1}\ln^{[i]}(u+e^{[m]})\right)^{-1}\\
&\ \ +\int_\Omega \left(\ln^{[m]}(u+e^{[m]})+\left(\prod_{i=1}^{m-1}\ln^{[i]}(u+e^{[m]})\right)^{-1}\right)  f(u,w).
\end{split}
\label{cz2.51ssddjjuiiid14114}
\end{equation}
In the sequel, we wish to control the taxis-involving integrals appearing on the right-hand sides of \dref{cz2.51ssddjjuiiid14114}. We shall proceed with the first alternative of \dref{thm-con}, since the second  alternative is included in \textbf{Case II} below.  Then, from
 the definition of $\mu_r$ in \dref{1.3xcxddfff29} and the first case of condition \dref{thm-con}, we can easily infer that $\mu_m=\mu_{r+1}=+\infty$, and so, by the local boundedness of $f$ due to \eqref{f-source}, we find  there exists a positive constant $ f_0$ such that
\begin{equation}
f(s,w)\leq f_0-\frac{(\chi+1) s^2}{\prod_{i=1}^{m}\ln^{[i]}(s+e^{[m]})},\ \ \  \forall (s, w)\in\left(0, \infty\right)\times \left(0, \max_{x\in\bar{\Omega}}w_0(x)\right).
\label{zjscz2.5297x96ssddd302222114}
\end{equation}
Noticing  that $\ln^{[i]}(e^{[m]}) >1 (i=1,2,\cdots m-1)$,  we deduce  from \dref{zjscz2.5297x96ssddd302222114} that
\begin{equation}
\begin{split}
&\chi u^2\left(\prod_{i=1}^{m-1}\ln^{[i]}(u+e^{[m]})\right)^{-1}+\kappa\xi u\left(\prod_{i=1}^{m-1}\ln^{[i]}(u+e^{[m]})\right)^{-1}\\
&\ \ + \left(\ln^{[m]}(u+e^{[m]})+\left(\prod_{i=1}^{m-1}\ln^{[i]}(u+e^{[m]})\right)^{-1}\right)  f(u,w)\\
&\leq - u^2\left(\prod_{i=1}^{m-1}\ln^{[i]}(u+e^{[m]})\right)^{-1}+\kappa\xi u+f_0\\
&\ \ +f_0 \ln^{[m]}(u+e^{[m]})-(\chi+1) u^2\left(\prod_{i=1}^{m-1}\ln^{[i]}(u+e^{[m]})\right)^{-2}\left(\ln^{[m]}(u+e^{[m]})\right)^{-1}\\
&\leq - u^2\left(\prod_{i=1}^{m-1}\ln^{[i]}(u+e^{[m]})\right)^{-1}+\tilde{f}_0\\
&\leq - (u+e^{[m]})\ln^{[m]}(u+e^{[m]})+\hat{f}_0=-g(u)+\hat{f}_0,
\end{split}
\label{2333cz2.51ssddddffjjdfffuiiid14114}
\end{equation}
where $\tilde{f}_0$ and $\hat{f}_0$ are finite numbers and are given respectively  by
\begin{align*}
\tilde{f}_0=\sup_{s>0}\Bigr\{&\kappa\xi s+f_0+f_0 \ln^{[m]}(s+e^{[m]})\\
 &(\chi+1) s^2\left(\prod_{i=1}^{m-1}\ln^{[i]}(s+e^{[m]})\right)^{-2}\left(\ln^{[m]}
 (s+e^{[m]})\right)^{-1}\Bigr\}<+\infty
 \end{align*}
and
$$
\hat{f}_0=
\sup_{s>0}\left\{- s^2\left(\prod_{i=1}^{m-1}\ln^{[i]}(s+e^{[m]})\right)^{-1}+\tilde{f}_0
+(s+e^{[m]})\ln^{[m]}(s+e^{[m]})\right\}<+\infty.
$$
 Combining \dref{cz2.51ssddjjuiiid14114} with \dref{2333cz2.51ssddddffjjdfffuiiid14114} and recalling $g^{''}(u)>0$ and the boundedness of $\Omega$, we readily derive an ODI for $g(u)$ as follows:
$$
\frac{d}{dt}\int_{\Omega}g(u)+\int_{\Omega}g(u)\leq C_1,  \ \ \forall t\in(0, T_m),
\label{cz2.51ssddjjuiiid1sdfff4114}
$$
entailing trivially   that
\begin{equation}
 \int_{\Omega}g(u)\leq C_2, \ \ \forall t\in(0, T_m).
\label{cz2.51ssddjjuiiid1sdfjjkff4114}
\end{equation}

\textbf{Case II:  $\tau>0$.} We multiply  the second  equation in \dref{1.1}
  by $-\Delta v $, integrating over $\Omega$ and using  the Young inequality to obtain
\begin{equation}
 \frac{\tau}{{2}}\frac{d}{dt}\int_\Omega |\nabla v |^2+\int_\Omega |\Delta v |^2+\int_{\Omega}|\nabla v |^2
= -\int_\Omega  u \Delta v, \ \ \  \forall  t\in(0, T_m).
\label{cz2.51dfgggssddjjuiiid14114}
\end{equation}
Combining  \dref{cz2.51dfgggssddjjuiiid14114} with \dref{key-id-ulnu}, using the $(L^1, L^2)$-bound  of $(u,v)$ in Lemma \ref{ul1-bdd},  the pointwise estimate of $-\Delta w$ in \dref{x1.731426677gghh} and Young's inequality  with epsilon, for any $\epsilon\in (0,\chi)$, we have
\begin{equation}
\begin{split}
&\frac{d}{dt}\int_{\Omega}\left[u\ln u+\frac{\tau\chi}{{2}}|\nabla v |^{{2}}\right]+4\int_\Omega |\nabla u^\frac{1}{2}|^2+\chi\int_{\Omega}|\Delta v |^2+\chi\int_{\Omega}|\nabla v |^2
\\
&=-2\chi\int_\Omega  u \Delta v-\xi\int_\Omega    u\Delta w
  + \int_\Omega  (1+\ln u)f(u,w)\\
&\leq 2\chi\int_\Omega  u |\Delta v|+\xi\int_\Omega    u\left(\tau \|w_0\|_{L^\infty(\Omega)} v+\kappa\right)
  +\int_\Omega  (1+\ln u)f(u,w) \\
&\leq\chi\int_{\Omega}|\Delta v |^2+(\chi+\epsilon)\int_\Omega  u^2+\frac{\tau^2\xi^2\|w_0\|_{L^\infty}^2}{4\epsilon}\int_\Omega   v^2+\kappa\xi\int_\Omega    u+
\int_\Omega   (1+\ln u)f(u,w)\\
&\leq \chi\int_{\Omega}|\Delta v |^2+
\int_\Omega \left[(\chi+\epsilon)  u^2 + (1+\ln u)f(u,w)\right]+C_\epsilon.
\end{split}
\label{ssddcz2.51ssddjjdfffuiiid14114}
\end{equation}
Now, we are almost  in the same situation as \cite{Xiang18-JMP};  for convenience, we present a  short argument here:  by the definition of   $\mu_1$ in \dref{1.3xcxddfff29}, we find  there exists a constant $ s_\epsilon>1$ such that
\begin{equation}
f(s,w)\leq -(\mu_1-\epsilon)\frac{ s^2}{\ln s},\ \ \  \forall (s, w)\in\left(s_\epsilon, \infty\right)\times \left(0, \max_{x\in\bar{\Omega}}w_0(x)\right).
\label{f-grow}
\end{equation}
where $\mu_1$ is understood as $\chi+1$ in the case of $\mu_1=+\infty$ (We here remark that $\mu_1=0$ is quite possible, which is the case, in particular, when $f\equiv0$). Then,  by  \dref{f-grow}, \eqref{f-source} and the boundedness of $\Omega$, we readily conclude  there exists  $C_\epsilon>0$ such that
\begin{equation}
 \int_\Omega \left[(\chi+\epsilon) u^2+  (1+\ln u)f(u,w)\right]\leq \left[(\chi-\mu_1)^+ +3\epsilon\right]\int_\Omega u^2+C_\epsilon.
\label{cz2.51ssddddfjfiiid14114}
\end{equation}
Inserting \dref{cz2.51ssddddfjfiiid14114} into \dref{ssddcz2.51ssddjjdfffuiiid14114}, we  end up with
\begin{equation}
\frac{d}{dt}\int_{\Omega}\left[u\ln u+\frac{\tau\chi}{{2}}|\nabla v |^{{2}}\right]+\chi\int_{\Omega}|\nabla v |^2+4\int_\Omega \nabla u^{\frac{1}{2}}|^2\leq \left[(\chi-\mu_1)^+ +3\epsilon\right]\int_\Omega u^2+C_\epsilon.
\label{ssddcz2.51ssddjjdfffuisdfffgiid14114}
\end{equation}
The  2D G-N inequality (c.f. Lemma \ref{GN-inter}) along with the $L^1$-boundedness of $u$ in \eqref{ul1-est-sub} yields
\begin{equation}
\begin{split}
 \int_\Omega u^2= |u^{\frac{1}{2}}\|^{4}_{L^4}&\leq   C_{GN}^4 \left(\|\nabla u^{\frac{1}{2}}\|^\frac{1}{2}_{L^2}\|u^{\frac{1}{2}}\|^\frac{1}{2}_{L^2}+\|u^{\frac{1}{2}}
 \|_{L^2}\right)^4\\
&\leq 8C_{GN}^4\left(M_1\|\nabla u^{\frac{1}{2}}\|^{2}_{L^2(\Omega)}+M_1^2\right).
\end{split}
\label{cz2.51ssdsdddddddfjfiiid14114}
\end{equation}
Next, since
$$u\ln u\leq \epsilon u^2+ L_\epsilon, \ \ \ L_\epsilon= \sup_{s>0}\{s\ln s-
\epsilon s^2\}<+\infty, $$
we thus get from  \dref{cz2.51ssdsdddddddfjfiiid14114} and  \dref{ssddcz2.51ssddjjdfffuisdfffgiid14114}  that
\begin{equation}
\begin{split}
\frac{d}{dt}\int_{\Omega}\left[u\ln u+\frac{\tau\chi}{2}|\nabla v |^2\right]&+\int_\Omega u\ln u+\chi\int_{\Omega}|\nabla v |^2\\
&+4\left(1-2M_1C_{GN}^4\left[(\chi-\mu_1)^+ +4\epsilon\right]\right)\int_\Omega |\nabla u^\frac{1}{2}|^{2}\\
&\leq 8M_1^2C_{GN}^4\left[(\chi-\mu_1)^+ +4\epsilon\right]+L_\epsilon+C_\epsilon.
\end{split}
\label{ssddcz2.51ssddjjdfffuihhhhhhhhjjjsdfffgiid14114}
\end{equation}
Now, due to the second alternative of  \dref{thm-con}, we fix, for instance,
$$
\epsilon=\frac{1-2M_1C_{GN}^4(\chi-\mu_1)^+}{8}>0£¬
$$
 in \dref{ssddcz2.51ssddjjdfffuihhhhhhhhjjjsdfffgiid14114} and apply a couple of elementary manipulations to conclude  that
$$
\int_{\Omega}\left[u\ln u+\frac{\tau\chi}{{2}}|\nabla v |^{{2}}\right]\leq C_3,  \quad \quad \forall t\in(0, T_m),
$$
 which along with the fact that $-s \ln s\leq e^{-1}$ for all $s > 0$ further entails
\begin{equation}
\int_\Omega\left[|u\ln u|+\frac{\tau\chi}{{2}}|\nabla v |^{{2}}\right]\leq C_4,  \quad \quad \forall t\in(0, T_m).
\label{cz2.51ssddjjuiddfffiid1sddfggdfjjkff4114}
\end{equation}
 The desired estimate \dref{zjscz2.5297x96302222114} follows from \dref{cz2.51ssddjjuiddfffiid1sddfggdfjjkff4114}, \dref{cz2.51ssddjjuiiid1sdfjjkff4114} and the definition of $g$ in \dref{x1.73sdfghh1426677gg}.
\end{proof}
In the case of $\tau>0$, upon the obtainment of the boundedness in \dref{cz2.51ssddjjuiddfffiid1sddfggdfjjkff4114}, in 2D setting, using the well-known procedure, cf. \cite{Osakix391, Xiang15-JDE, Xiang18-JMP},  we can easily obtain the $L^\infty$-boundedness of $u$ and then the claimed boundedness in Theorem \ref{theorem3}. In the case of $\tau=0$, we shall show that the boundnedness information in \dref{zjscz2.5297x96302222114} will also be sufficient to derive our desired boundedness as announced in Theorem \ref{theorem3}. For this purpose, we need the following generalization of the logarithmic version of  Gagliardo-Nirenberg inequality  \cite[Lemma A.5] {Tao79477}, whose idea was initially  demonstrated in \cite{Biler555216}.

\begin{lemma}\label{lemma45630jjjj223}
 Let $\Omega\subset \mathbb{R}^2$ be a bounded domain with smooth boundary, and let $q\in(1,\infty)$,
$r\in (0, q)$. If $\frac{|g(s)|}{s}$ is nondecreasing in $(s_0,+\infty)$ (for some $s_0>1$)
 and
 $\lim_{s\rightarrow+\infty}\frac{|g(s)|}{s}=+\infty$,
then there exists $C >0$ such that for each $\varepsilon >0$ one can pick $ C_{\varepsilon}  > 0$ with the property that
\begin{equation}\|\varphi\|_{L^{q}(\Omega)}^q \leq \varepsilon\|\nabla \varphi\|_{L^{2}(\Omega)}^{q-r}\|g(\varphi)\|^{r}_{L^{r}(\Omega)}+C\|\varphi\|_{L^{r}(\Omega)}^q+  C_{\varepsilon}
\label{1sdrfgggg.3xcssddffx29}
\end{equation}
holds for all $\varphi\in W^{1,2}(\Omega).$
\end{lemma}
\begin{proof}
According to the Gagliardo-Nirenberg inequality, there exists $C_1 > 0$ such that
\begin{equation}\|\psi\|_{L^{q}(\Omega)}^q \leq C_1\|\nabla \psi\|_{L^2(\Omega)}^{q-r}\|\psi\|^{r}_{L^r(\Omega)}+C_1\|\psi\|_{L^r(\Omega)}^q, \  \ \ \forall\psi\in W^{1,2}(\Omega).
\label{cz2.5g556789ssdfffhhjui78jj90099}
\end{equation}
Since $\lim_{s\rightarrow+\infty}\frac{|g(s)|}{s}=+\infty$, for any $\varepsilon > 0$, we can choose $\lambda=\lambda(\varepsilon) >s_0$ large enough fulfilling
\begin{equation}\frac{2^{2q-r}C_1\lambda^r}{|g(\lambda)|^r }< \varepsilon.
\label{cz2.5g556789ssdfffhsdfffffhjui78jj90099}
\end{equation}
Next, define
\begin{equation*}
\alpha(s)= \left\{\begin{array}{ll}
0, &\mbox{if}~~|s|\leq\lambda,\\
2(|s|-\lambda), &\mbox{if}~~\lambda<|s|<2\lambda,\\
|s|, &\mbox{if}~~|s|\geq2\lambda.
 \end{array}\right.
\end{equation*}
Then we see $\alpha\in W^{1,\infty}_{loc} (\mathbb{R})$, $0 \leq \alpha(s)\leq |s|$ and $|\alpha'(s)| \leq2$  for a.e. $s\in \mathbb{R}$.
 Hence,
 $$\|\alpha(\varphi)\|_{L^r}^q\leq \|\varphi\|_{L^r}^q~~~\mbox{and}~~~\|\nabla\alpha(\varphi)\|_{L^2}^{q-r}\leq 2^{q-r} \|\nabla\varphi\|_{L^2}^{q-r}, \ \ \ \forall \varphi\in W^{1,2}(\Omega).$$
 Moreover, since $\frac{|g(s)|}{s}$ is nondecreasing on  $(s_0,+\infty)$, we infer
\begin{align*}
 \|\alpha(\varphi)\|_{L^r(\Omega)}^r&= \int_{\{|\varphi|>\lambda\}} |\alpha(\varphi)|^r\\
 &\leq \int_{\{|\varphi|>\lambda\} }|\varphi|^r \\
&=\int_{\{|\varphi|>\lambda\}} \left|\frac{\varphi}{g(\varphi)}g(\varphi)\right|^r \\
&\leq  \left(\frac{\lambda}{|g(\lambda)|}\right)^r\int_{\{|\varphi|>\lambda\}}|g(\varphi)|^r \\
& \leq  \left(\frac{\lambda}{|g(\lambda)|}\right)^r\left\|g(\varphi)\right\|^r_{L^r(\Omega)} .
\end{align*}
Next, it follows from $0\leq |s|-\alpha(s)\leq 2\lambda$ on $\mathbb{R}$ that
$$\left\||\varphi|-\alpha(\varphi)\right\|^q_{L^q(\Omega)}\leq (2\lambda)^q|\Omega|.$$
In view of the elementary inequality $(a+b)^q\leq 2^q(a^q +b^q )$  for all nonnegative $a$ and $b$,  we thus deduce from \dref{cz2.5g556789ssdfffhhjui78jj90099} and \dref{cz2.5g556789ssdfffhsdfffffhjui78jj90099}  that
\begin{align*}
\|\varphi\|_{L^q}^q \leq &2^q\|\alpha(\varphi)\|_{L^q}^q+2^q\||\varphi|-\alpha(\varphi)\|_{L^q}^q \\
&\leq  2^qC_1\left[\|\nabla\alpha(\varphi)\|_{L^2}^{q-r}\|\alpha(\varphi)\|_{L^r}^{r}
+\|\alpha(\varphi)\|_{L^r}^{q}\right]+2^q\left\||\varphi|-\alpha(\varphi)\right\|_{L^q}^q \\
&\leq C_12^{2q-r}\|\nabla\varphi\|_{L^2}^{q-r}\left(\frac{\lambda}{|g(\lambda)|}\right)^r
\left\|g(\varphi)\right\|^r_{L^r}+
2^qC_1\|\varphi\|_{L^r}^{q}+2^q(2\lambda)^q|\Omega| \\
&\leq \varepsilon \|\nabla\varphi\|_{L^2}^{q-r}
\left\|g(\varphi)\right\|^r_{L^r}+2^qC_1\|\varphi\|_{L^r}^{q}+2^q(2\lambda)^q|\Omega|.
\end{align*}
In light of our choice of $\lambda$, this entails \dref{1sdrfgggg.3xcssddffx29} by choosing  $C := 2^qC_1$ and $C_\varepsilon=2^q(2\lambda)^q|\Omega|$.
\end{proof}
\begin{corollary}
For any $m\geq1$,  $$g(s)=(s+e^{[m]}) \ln^{[m]}(s+e^{[m]}),\ \ s>0.$$
 It is evident to see that  $g>0$ on $(0,+\infty)$,
 $$\frac{g(s)}{s} \  \mbox{is nondecreasing on} \  (1,+\infty) \text{ and } \lim_{s\rightarrow+\infty}\frac{g(s)}{s}=+\infty.$$
\end{corollary}
\begin{lemma}\label{lemma4ssdffff5630223}  There exists   $C>0$ such that
 the corresponding solution of \dref{1.1} satisfies
\begin{equation}\label{ul2-bdd}
 \tau\int_\Omega|\nabla v (\cdot,t)|^{{4}}+ \int_\Omega u ^{{2}}(\cdot,t)\leq C, \quad \forall  t\in(0,T_m).
\end{equation}
\end{lemma}
\begin{proof}
Multiplying both sides of the first equation in \dref{1.1} by $u$, integrating over $\Omega$  by parts and  applying \dref{x1.731426677gghh},  for $t\in(0, T_m)$, we arrive at
\begin{equation}
\begin{split}
& \frac{1}{{2}}\frac{d}{dt}\int_\Omega u^2+\int_{\Omega}|\nabla u|^2
\\
&=-\chi\int_\Omega \nabla\cdot( u\nabla v)
   u -\xi\int_\Omega \nabla\cdot( u\nabla w)
   u  +\int_\Omega    u f(u,w)  \\
&=\chi\int_\Omega  u\nabla u\nabla v
 +\xi\int_\Omega  u\nabla u\nabla w
   +\int_\Omega    u f(u,w)\\
&=-\frac{\chi}{2}\int_\Omega  u^{2}\Delta v
  -\frac{\xi}{2}\int_\Omega  u^{2}\Delta w
   +
\int_\Omega    u f(u,w)\\
&\leq -\frac{\chi}{2}\int_\Omega  u^{2}\Delta v
  +\frac{\xi}{2}\int_\Omega  u^{2} \left(\tau \|w_0\|_{L^\infty} v+\kappa\right)
   +\int_\Omega    u f(u,w).
\end{split}
\label{cz2aasweee.5114114}
\end{equation}
\textbf{Case I $\tau=0:$}  We substitute $-\Delta v=u-v$ by  \dref{1.1} into \dref{cz2aasweee.5114114} to get that
\begin{equation}
\begin{array}{rl}
&\disp{\frac{1}{2}\frac{d}{dt}\int_\Omega u^2+\int_{\Omega}|\nabla u|^2\leq\frac{\chi}{2}\int_\Omega  u^{3}
  +\frac{\xi\kappa}{2}\int_\Omega  u^{2}
   +
\int_\Omega    u f(u,w)}.
\end{array}
\label{cz2aasweee.ssdddkkllll5ddfgg114114}
\end{equation}
Next, since \dref{1.3xcxddfff29} together with the first case of \dref{thm-con}  implies that
$$C_1:=\sup_{s>0}\left\{ s f(s,w)+\frac{(\xi\kappa+1)}{2}s^{2} \right\}<+\infty,$$
so that we infer from  \dref{cz2aasweee.ssdddkkllll5ddfgg114114} that
\begin{equation}
\begin{array}{rl}
&\disp{\frac{d}{dt}\int_\Omega u^2+\int_\Omega  u^{2}+\int_{\Omega}|\nabla u|^2\leq\chi\int_\Omega  u^{3}+2C_1|\Omega|.
  }
\end{array}
\label{cz2aasweee.ssdddsdfggkkllll5ddfgg114114}
\end{equation}
For the integral on the  right-hand side, with $g$ defined by \dref{x1.73sdfghh1426677gg}, using the  estimates obtained in  Lemmas \ref{ul1-bdd} and \ref{lemma45630223},  we deduce from  Lemma \ref{lemma45630jjjj223} and its corollary  that
\begin{equation}
\begin{array}{rl}
\disp\chi \int_\Omega  u^{3}=&\disp{\chi\|u\|_{L^3(\Omega)}^{3}}\\
\leq&\disp{\chi\varepsilon \|\nabla u\|_{L^2(\Omega)}^{2}\|g(u)\|_{L^1(\Omega)}+C\chi\|u\|_{L^1(\Omega)}^3+
  C_{\varepsilon} \chi}\\
\leq&\disp{ \chi\varepsilon \|\nabla u\|_{L^{2}(\Omega)}^{2}C+C\chi M_1^3+
  C_{\varepsilon} \chi}\\
  \leq &\|\nabla u\|_{L^{2}(\Omega)}^2+C_2
\end{array}
\label{cz2.563022222ikopl2sdfssfggg44}
\end{equation}
by  picking  sufficiently small $\varepsilon$.  Then an ODI  for $\|u\|_{L^2}^2$ follows easily from \dref{cz2aasweee.ssdddsdfggkkllll5ddfgg114114} and \dref{cz2.563022222ikopl2sdfssfggg44}:
$$
\frac{d}{dt}\int_\Omega u^2+\int_\Omega  u^{2}\leq C_3,
$$
 which, upon being solved, yields readily the $L^2$-boundedness of $u$, as desired in \eqref{ul2-bdd}.

\textbf{Case II $\tau>0$:}  Notice that  $2\nabla v\cdot\nabla\Delta v = \Delta|\nabla v|^2-2|D^2v|^2$, by a straightforward computation using the
second equation in \dref{1.1} and  integrations by parts, we see  that
\begin{equation}
\begin{split}
 &\frac{\tau}{{2}}\frac{d}{dt}\int_\Omega |\nabla v |^4+\int_{\Omega}|\nabla|\nabla v|^2 |^2+2\int_{\Omega}|\nabla v|^2|D^2 v |^2+2\int_{\Omega}|\nabla v |^4\\
=&\int_{\partial\Omega}  |\nabla v|^2\frac{\partial|\nabla v|^2}{\partial\nu}-2\int_{\Omega}u\Delta v|\nabla v |^2
-2\int_{\Omega}u\nabla v\nabla|\nabla v |^2.
\end{split}
\label{cz2.51dfgggssddjddffjuiissdddsddid14114}
\end{equation}
Then we deduce from \dref{cz2aasweee.5114114}, \dref{cz2.51dfgggssddjddffjuiissdddsddid14114} and the Young's inequality   that
\begin{equation}
\begin{split}
&\frac{1}{2}\frac{d}{dt} \int_\Omega \left(u^2+ \tau |\nabla v|^4\right)+\int_{\Omega}|\nabla u|^2+\int_{\Omega}|\nabla|\nabla v|^2 |^2+2\int_{\Omega}|\nabla v|^2|D^2 v |^2+2\int_{\Omega}|\nabla v |^4\\
\leq &\int_{\partial\Omega}  |\nabla v|^2\frac{\partial|\nabla v|^2}{\partial\nu}-2\int_{\Omega}u\Delta v|\nabla v |^2
-2\int_{\Omega}u\nabla v\nabla|\nabla v |^2\\
&+\chi\int_\Omega  u\nabla u\nabla v
  +\frac{\xi}{2}\int_\Omega  u^{2} (\tau \|w_0\|_{L^\infty} v+\kappa)
   +\int_\Omega    u f(u,w) \\
\leq & \int_{\partial\Omega}  |\nabla v|^2\frac{\partial|\nabla v|^2}{\partial\nu}-2\int_{\Omega}u\Delta v|\nabla v |^2
+2\int_{\Omega}u^2|\nabla v|^2+\frac{1}{2}\int_{\Omega}|\nabla|\nabla v |^2|^2 +\frac{1}{2}\int_\Omega  |\nabla u|^2  \\
&+\frac{\chi^2}{2}\int_\Omega  u^{2}|\nabla v|^2
  +\frac{4}{5}\int_\Omega  u^{\frac{5}{2}}+\frac{1}{5}(\frac{\xi}{2})^5\left(\tau \|w_0\|_{L^\infty}\right)^5\int_\Omega   v^5+\frac{\xi\kappa}{2}\int_\Omega  u^{2} +\int_\Omega    u f(u,w).  \end{split}
\label{cz2.51dfgggssddjddffjusdfffiiidsddd141sdfff14}
\end{equation}
Next, by \dref{1.3xcxddfff29} and the second case of \dref{thm-con},  we obtain that
\begin{equation}C_4:=\sup_{s>0}\left\{ s f(s,w)+\frac{(1+\xi\kappa)}{2}s^2+\frac{4}{5}s^{\frac{5}{2}}-s^3 \right\}<+\infty.
\label{cz2.51dfgggssddjddffjusdfffiiidsddd141sdfffssddff14}
\end{equation}
 In light of the  $W^{1,2}$-boundedness of $v$ established in \dref{vl2-est-sub} and  \dref{zjscz2.5297x96302222114} and the 2D Sobolev embedding $W^{1,2}(\Omega)\hookrightarrow L^5(\Omega)$, we get the $L^5$-boundedness of $v$ on $(0,T_m)$. Hence, joining  \dref{cz2.51dfgggssddjddffjusdfffiiidsddd141sdfffssddff14} with  \dref{cz2.51dfgggssddjddffjusdfffiiidsddd141sdfff14}, we arrive at
 \begin{equation}
\begin{split}
\frac{1}{2}\frac{d}{dt} \int_\Omega \left(u^2+ \tau |\nabla v|^4\right)&+\frac{1}{2} \int_\Omega u^2+\frac{1}{2}\int_{\Omega}|\nabla u|^2+\frac{1}{2}\int_{\Omega}|\nabla|\nabla v|^2 |^2\\
&+2\int_{\Omega}|\nabla v|^2|D^2 v |^2+2\int_{\Omega}|\nabla v |^4\\
\leq & \int_{\partial\Omega}  |\nabla v|^2\frac{\partial|\nabla v|^2}{\partial\nu}-2\int_{\Omega}u\Delta v|\nabla v |^2\\
&+(2+\frac{\chi^2}{2})\int_\Omega  u^{2}|\nabla v|^2
   + \int_\Omega    u^3+C_4|\Omega|.\end{split}
\label{cz2.51dfgggssddjddfsdffffjsdffusdfffiiidsddd141sdfff14}
\end{equation}
We are at the same situation as we have in \cite{Xiang18-JMP}: Given the boundedness of $\|\nabla  v \|_{L^2}^2$, it is well-known that (cf. \cite{Ishida, Tao41215, Xiang18-JMAA}) the boundary trace embedding  implies that
\begin{equation} \label{com-est-4}
\begin{split}
\int_{\partial\Omega}  |\nabla v|^{2}\frac{\partial}{\partial \nu} |\nabla v|^2&\leq \epsilon \int_\Omega |\nabla |\nabla v|^2|^2+C_\epsilon \Bigr(\int_{\Omega} |\nabla v|^2\Bigr)^2\\
&\leq \epsilon \int_\Omega |\nabla |\nabla v|^2|^2+C_\epsilon,\quad \quad  \forall \epsilon>0.
\end{split}
\end{equation}
Next, since $|\Delta  v | \leq\sqrt{2}|D^2 v |$, by the Young inequality, we estimate, for any $\epsilon>0$,  that
\begin{equation}
\begin{split}
&-2\int_{\Omega}u\Delta v|\nabla v |^2+(2+\frac{\chi^2}{2})\int_\Omega  u^{2}|\nabla v|^2
   + \int_\Omega    u^3\\
&\leq \int_\Omega  |\nabla  v |^{2}|D^2 v |^2+(4+\frac{\chi^2}{2})\int_\Omega  u^2  |\nabla  v |^{2}+ \int_\Omega    u^3\\
&\leq \int_\Omega  |\nabla  v |^{2}|D^2 v |^2+\epsilon\int_{\Omega}  |\nabla v|^6+\left[\frac{2}{3}(3\epsilon)^{-\frac{1}{2}}(4+\frac{\chi^2}{2})^{\frac{3}{2}}+1\right]\int_\Omega u^3.
\end{split}
\label{444cz2.5ghju48hjuikl1}
\end{equation}
From the boundednedd of $\|\nabla v\|_{L^2}$,  we use   the 2D G-N inequality to derive that
\begin{equation}
\begin{split}
 \int_{\Omega}  |\nabla v|^6= \||\nabla v|^2\|^3_{L^3}&\leq C_5 \|\nabla |\nabla  v |^{2}\|_{L^2}^{2}\| |\nabla  v |^2\|_{L^1}+C_5\|   |\nabla v |^2\|_{L^1}^{3}\\
&\leq C_6 \int_{\Omega}|\nabla |\nabla  v |^2 +C_6.
\end{split}
\label{cz2.51dfgggssjjjjjddjddfsdsdffggfffssdddddfjsdffusdfffiiidsddd141sdfff14}
\end{equation}
For the integral involving $\int_{\Omega}u^3$, based on the boundedness $\|u\|_{L^1}+\|u\ln u\|_{L^1}$ as ensured in Lemmas \ref{ul1-bdd} and \ref{lemma45630223},  we easily infer from the generalized  G-N inequality in Lemma \ref{lemma45630jjjj223}  that
\begin{equation} \label{ul3-bdd by}
\int_\Omega  u^3\leq \eta \int_\Omega |\nabla u |^2+C_\eta, \quad \quad \forall \eta>0.
\end{equation}
Combining  the estimates \dref{com-est-4},   \dref{444cz2.5ghju48hjuikl1}, \dref{cz2.51dfgggssjjjjjddjddfsdsdffggfffssdddddfjsdffusdfffiiidsddd141sdfff14} and \dref{ul3-bdd by}  with \dref{cz2.51dfgggssddjddfsdffffjsdffusdfffiiidsddd141sdfff14} and choosing sufficiently small  $\epsilon>0$ and $\eta>0$,  we obtain an ODI as follows:
$$
\frac{d}{dt} \int_\Omega \left(u^2+ \tau|\nabla v|^4\right)+ \min\left\{1,\frac{4}{\tau}\right\}\int_\Omega \Bigr(u^2+ \tau |\nabla v|^4\Bigr)\leq C(u_0, v_0, |\Omega|, \chi, \xi, \tau, f),
$$
which directly yields the uniform boundedness of $\|u\|_{L^2}+\|\nabla v\|_{L^4}$, as desired in \eqref{ul2-bdd}.
\end{proof}

\subsection{From $L^2$ to $L^\infty$: The proof of Theorem \ref{theorem3}}
The proof now becomes rather standard. Thanks to the    $L^2$-boundedness of $u$ in \eqref{ul2-bdd}, if $\tau>0,$ we infer from  the known smoothing $L^p$-$L^q$-estimates for the Neumann heat semigroup  $\{e^{t\Delta}\}_{t\geq0}$ in $\Omega$ (cf.   \cite{Cao15, Horstmann791, Winkler792}) to the semigroup representation of the $v$-equation in \dref{1.1} that
$$
\| v(\cdot, t)\|_{W^{1,q}}\leq C_1\left(1+\sup_{s\in(0,t)}\|u(\cdot, s)\|_{L^2}\right)\leq C_2, \ \ \ \forall t\in(0, T_m)~~\mbox{and }~~q\in(1,+\infty).
$$
While, if $\tau=0$,  the standard $W^{2,p}$-regularity theory  (see e.g. \cite{Ladyzenskaja710}) to the second equation in \dref{1.1} implies the $W^{2,2}$-boundedness of $v(\cdot, t)$, and hence  the Sobolev embeddings $W^{2,2}(\Omega) \hookrightarrow W^{1,q}(\Omega)\hookrightarrow L^{\infty}(\Omega)$ (for all $q\in(2,+\infty)$) in two-dimensional space entail that
\begin{equation}
\begin{array}{rl}
\| v(\cdot, t)\|_{L^{\infty}(\Omega)}+\| v(\cdot, t)\|_{W^{1,q}(\Omega)}\leq C_3, ~~ \forall   t\in(0,T_m)~~\mbox{and }~~q\in(2,+\infty).
\end{array}
\label{cz2sedfgg.5g5gllllghh56789hhjui7ssddd8jj90099}
\end{equation}
Multiplying both sides of the first equation in \dref{1.1} by $3u^{2}$, integrating over $\Omega$ by parts and applying the pointwise boundedness of $\Delta w$ in \eqref{x1.731426677gghh}, the $L^2$-boundedness of $u$ and \dref{cz2sedfgg.5g5gllllghh56789hhjui7ssddd8jj90099},  Young's inequality with epsilon and the 2D G-N inequality, we arrive at
\begin{equation}
\begin{split}
&\frac{d}{dt}\int_\Omega u^3+\int_\Omega u^3
+6\int_{\Omega}u |\nabla u|^2
\\
=&6\chi\int_\Omega  u^2\nabla u\cdot\nabla v
   -2\xi\int_\Omega  u^2\Delta w
   +\int_\Omega u^3+
3\int_\Omega   u^2f(u,w) \\
\leq& 6\chi\int_\Omega  u^2\nabla u\cdot\nabla v
   +2\xi\int_\Omega  u^2(\tau \|w_0\|_{L^\infty} v+\kappa)
   +\int_\Omega   \left[u^3+3u^2f(u,w)\right]  \\
\leq&6\chi\int_\Omega  u^2\nabla u\cdot\nabla v
   +\xi C_4\int_\Omega  u^2
   +\int_\Omega   \left[u^3+3u^2f(u,w)\right]\\
\leq & 3\int_\Omega  u|\nabla u|^2+3\chi^2\int_\Omega u^3|\nabla v|^2+\int_\Omega   \left[\xi C_4 u^2+u^3+3u^2f(u,w)\right]\\
\leq &3\int_\Omega  u|\nabla u|^2+3\chi^2\int_\Omega   u^4+\frac{3^4\chi^2}{4^4}\int_\Omega |\nabla v|^8+\int_\Omega   u^4 +C_5|\Omega|\\
\leq &3\int_\Omega  u|\nabla u|^2+(3\chi^2+1)\| u^\frac{3}{2}\|_{L^\frac{8}{3}}^\frac{8}{3}+C_6\\
\leq &3\int_\Omega  u|\nabla u|^2+(3\chi^2+1)C_7\left(\|\nabla  u^\frac{3}{2}\|_{L^2}^\frac{4}{3}\| u^\frac{3}{2}\|_{L^\frac{4}{3}}^\frac{4}{3}+\| u^\frac{3}{2}\|_{L^\frac{4}{3}}^\frac{8}{3}\right)+C_8\\
\leq&3\int_\Omega  u|\nabla u|^2+(3\chi^2+1)C_9 \|\nabla  u^\frac{3}{2}\|_{L^2}^\frac{4}{3}+C_{10}\\
\leq &4\int_\Omega  u|\nabla u|^2+C_{11},
\end{split}
\label{cz2sedfgg.5g5gljjjjkklllghh56789hhjui7ssddd8jj90099}
\end{equation}
where we have used  \eqref{f-source}, \dref{1.3xcxddfff29}  and \dref{thm-con} to see that
$$C_5:=\sup_{s>0}\left\{ s^{2} f(s,w)+\xi C_4s^2+s^{3} -s^4\right\}<+\infty.
$$
Then  \dref{cz2sedfgg.5g5gljjjjkklllghh56789hhjui7ssddd8jj90099} implies that
$$
\frac{d}{dt}\int_\Omega  u^3+\int_\Omega  u^3\leq C_{11},
$$
yielding trivially
\begin{equation}
\|u(\cdot, t)\|_{L^{{3}}(\Omega)}\leq C_{12}, \quad \forall  t\in(0,T_m).
\label{cz2aasweeeddfff.5ssedfssdffgddff114114}
\end{equation}
In light of  \dref{cz2aasweeeddfff.5ssedfssdffgddff114114} and the equation $\tau v_t=\Delta v-v+u$, we use the smoothing estimates for the Neumann heat semigroup (case of $\tau>0$) or  the $W^{2,p}$-estimate (case of $\tau=0$) and the
embedding $W^{2,3}(\Omega) \hookrightarrow W^{1,\infty}(\Omega)$  in two-dimensional space  to conclude that
\begin{equation}
\begin{array}{rl}
\| v(\cdot, t)\|_{W^{1,\infty}(\Omega)}\leq C_{13}, \quad \forall  t\in(0,T_m).
\end{array}
\label{cz2sedfgg.5g5gllllghh567ssdf89hhjui7ssdddssgg8jj90099}
\end{equation}
This coupled with the fact  that $w(x,t)=w(x,0)e^{-\int_0^t v(x,s)ds}$ gives rise to
\begin{equation}
\|\nabla w(\cdot,t)\|_{L^{\infty}(\Omega)}  \leq C_{14} , \quad \forall  t\in(0,T_m).
\label{cz2sedfgg.5g5ssdfffgggllllghh567ssdf89hhjui7ssdddssgg8jj90099}
\end{equation}
To derive the $L^\infty$-boundedness of $u$, based on the boundedness results obtained in \dref{cz2aasweeeddfff.5ssedfssdffgddff114114}, \dref{cz2sedfgg.5g5gllllghh567ssdf89hhjui7ssdddssgg8jj90099}
and \dref{cz2sedfgg.5g5ssdfffgggllllghh567ssdf89hhjui7ssdddssgg8jj90099} and our assumptions \dref{1.3xcxddfff29} and \dref{thm-con}, we use  the variation-of-constants formula for the  $u$-equation in \eqref{1.1}  and the well-known smoothing $L^p$-$L^q$-estimates for the Neumann heat  semigroup  $\{e^{t\Delta}\}_{t\geq0}$ to deduce that
\begin{equation}\begin{split}
\|u(t)\|_{L^\infty}&\leq \|e^{t\Delta }u_0\|_{L^\infty}+\chi\int_0^t\left\|e^{(t-s)\Delta }\nabla \cdot((u\nabla v)(s))\right\|_{L^\infty}ds \\
  &\quad + \xi\int_0^t\left\|e^{(t-s)\Delta }\nabla \cdot((u\nabla w)(s))\right\|_{L^\infty}ds+\int_0^t\left\|e^{(t-s)\Delta }f(u(s),w(s))\right\|_{L^\infty}ds\\
  \leq &\|u_0\|_{L^\infty}+C_{15}\chi\int_0^t\left(1+(t-s)^{-\frac{1}{2}-\frac{1}{3}}\right)e^{-\lambda_1(t-s) }\left\|(u\nabla v)(s)\right\|_{L^3}ds\\
  &+C_{16}\xi\int_0^t\left(1+(t-s)^{-\frac{1}{2}-\frac{1}{3}}\right)e^{-\lambda_1(t-s) }\left\|(u\nabla w)(s)\right\|_{L^3}ds\\
  &+C_{17}\int_0^t\left(1+(t-s)^{-\frac{2}{3}}\right)e^{-\lambda_1(t-s) }\left[1+\left\|u^2(s)\right\|_{L^\frac{3}{2}}\right]ds\\
 \leq & \|u_0\|_{L^\infty}+C_{15}\chi\int_0^t\left(1+(t-s)^{-\frac{1}{2}-\frac{1}{3}}\right)e^{-\lambda_1(t-s) }\left\|u\right\|_{L^3}\left\|\nabla v\right\|_{L^\infty}ds\\
  &+C_{16}\xi\int_0^t\left(1+(t-s)^{-\frac{1}{2}-\frac{1}{3}}\right)e^{-\lambda_1(t-s) }\left\|u\right\|_{L^3}\left\|\nabla w\right\|_{L^\infty}ds\\
  &+C_{17}\int_0^t\left(1+(t-s)^{-\frac{2}{3}}\right)e^{-\lambda_1(t-s) }\left[1+\left\|u(s)\right\|_{L^3}^2\right]ds\\
  \leq &\|u_0\|_{L^\infty}+C_{18}(\chi+\xi)\int_0^t\left(1+z^{-\frac{5}{6}}\right)e^{-\lambda_1z }dz+C_{19}\int_0^t\left(1+z^{-\frac{2}{3}}\right)e^{-\lambda_1z }dz\\
   \leq &\|u_0\|_{L^\infty}+C_{20}(1+\chi+\xi),  \quad \forall t\in(0, T_m).
  \end{split}
  \label{cz2.5g5gghh56789hhjui78jj90099}
  \end{equation}
Here, $\lambda_1(>0)$ is the first nonzero eigenvalue of $-\Delta$ under homogeneous Neumann boundary condition.  In view of \dref{cz2sedfgg.5g5gllllghh567ssdf89hhjui7ssdddssgg8jj90099} and \dref{cz2.5g5gghh56789hhjui78jj90099} and \eqref{1.163072x} of  Lemma \ref{lemma70}, we first see that $T_m=\infty$, and then,  the desired unform boundedness \dref{uvw-bdd} follows from \dref{cz2sedfgg.5g5gllllghh567ssdf89hhjui7ssdddssgg8jj90099}, \dref{cz2sedfgg.5g5ssdfffgggllllghh567ssdf89hhjui7ssdddssgg8jj90099} and \dref{cz2.5g5gghh56789hhjui78jj90099}; that is, the classical solution $(u,v,w)$ of \dref{1.1} is global in time and is uniformly bounded.

{\bf Acknowledgements}: The authors greatly thank  the two anonymous referees and  two editors for giving positive and valuable comments and suggestions from different perspectives, which further helped them to improve the exposition  of this work.  The research of T. Xiang   is  funded by  the NSF of China  (No. 11601516 and 11871226) and the  Research Funds of Renmin University of China (No. 2018030199). The research of J. Zheng is  supported by the Shandong Provincial
Science Foundation for Outstanding Youth (No. ZR2018JL005), the NSF of China (No. 11601215) and China Postdoctoral Science Foundation (No. 2019M650927, 2019T120168).

\end{document}